\documentclass[a4paper,10pt]{amsart}
\usepackage{amsmath,amsthm,amssymb,latexsym,enumerate,color,hyperref}
\usepackage{graphicx}
\usepackage{esint}
\usepackage{xcolor}

\numberwithin{equation}{section}

\usepackage{verbatim}

\begin{document}

\newtheorem{thm}{Theorem}[section]
\newtheorem{prop}[thm]{Proposition}
\newtheorem{lem}[thm]{Lemma}
\newtheorem{cor}[thm]{Corollary}
\newtheorem{rem}[thm]{Remark}
\newtheorem*{defn}{Definition}
\newtheorem*{note}{Note}

\newcommand{\DD}{\mathbb{D}}
\newcommand{\NN}{\mathbb{N}}
\newcommand{\ZZ}{\mathbb{Z}}
\newcommand{\QQ}{\mathbb{Q}}
\newcommand{\RR}{\mathbb{R}}
\newcommand{\CC}{\mathbb{C}}
\renewcommand{\SS}{\mathbb{S}}

\renewcommand{\theequation}{\arabic{section}.\arabic{equation}}

\newcommand{\supp}{\mathop{\mathrm{supp}}}    

\newcommand{\re}{\mathop{\mathrm{Re}}}   
\newcommand{\im}{\mathop{\mathrm{Im}}}   
\newcommand{\dist}{\mathop{\mathrm{dist}}}  
\newcommand{\spn}{\mathop{\mathrm{span}}}   
\newcommand{\ind}{\mathop{\mathrm{ind}}}   
\newcommand{\rank}{\mathop{\mathrm{rank}}}   
\newcommand{\Fix}{\mathop{\mathrm{Fix}}}   
\newcommand{\codim}{\mathop{\mathrm{codim}}}   
\newcommand{\conv}{\mathop{\mathrm{conv}}}   
\newcommand{\pa}{\partial}
\newcommand{\ve}{\varepsilon}
\newcommand{\zi}{\zeta}
\newcommand{\Si}{\Sigma}
\newcommand{\cA}{{\mathcal A}}
\newcommand{\cG}{{\mathcal G}}
\newcommand{\cH}{{\mathcal H}}
\newcommand{\cI}{{\mathcal I}}
\newcommand{\cJ}{{\mathcal J}}
\newcommand{\cK}{{\mathcal K}}
\newcommand{\cL}{{\mathcal L}}
\newcommand{\cN}{{\mathcal N}}
\newcommand{\cR}{{\mathcal R}}
\newcommand{\cS}{{\mathcal S}}
\newcommand{\cT}{{\mathcal T}}
\newcommand{\cU}{{\mathcal U}}
\newcommand{\cC}{\mathcal{C}}
\newcommand{\cM}{\mathcal{M}}
\newcommand{\De}{\Delta}
\newcommand{\cX}{\mathcal{X}}
\newcommand{\cP}{\mathcal{P}}
\newcommand{\ol}{\overline}
\newcommand{\ul}{\underline}
\newcommand{\AC}{\mathop{\mathrm{AC}}}   
\newcommand{\Lip}{\mathop{\mathrm{Lip}}}   
\newcommand{\es}{\mathop{\mathrm{esssup}}}   
\newcommand{\les}{\mathop{\mathrm{les}}}   
\newcommand{\la}{\lambda}
\newcommand{\La}{\Lambda}    
\newcommand{\de}{\delta}    
\newcommand{\fhi}{\varphi} 
\newcommand{\vro}{\varrho} 
\newcommand{\ga}{\gamma}    
\newcommand{\ka}{\kappa}   

\newcommand{\core}{\heartsuit}
\newcommand{\diam}{\mathrm{diam}}

\newcommand{\lan}{\langle}
\newcommand{\ran}{\rangle}
\newcommand{\tr}{\mathop{\mathrm{tr}}}
\newcommand{\diag}{\mathop{\mathrm{diag}}}
\newcommand{\dv}{\mathop{\mathrm{div}}}

\newcommand{\al}{\alpha}
\newcommand{\be}{\beta}
\newcommand{\Om}{\Omega}
\newcommand{\na}{\nabla}

\newcommand{\nr}{\Vert}
\newcommand{\weak} {\rightharpoonup}
 
\newcommand{\om}{\omega}
\newcommand{\si}{\sigma}
\newcommand{\te}{\theta}
\newcommand{\Ga}{\Gamma}

\title[Interpolating estimates]{Interpolating estimates with applications to some quantitative symmetry results}

\author{Rolando Magnanini} 
\address{Dipartimento di Matematica ed Informatica ``U.~Dini'',
Universit\` a di Firenze, viale Morgagni 67/A, 50134 Firenze, Italy.}
    \email{rolando.magnanini@unifi.it}
    \urladdr{http://web.math.unifi.it/users/magnanin}

\author{Giorgio Poggesi}
\address{Department of Mathematics and Statistics, The University of Western Australia, 35 Stirling Highway, Crawley, Perth, WA 6009, Australia}
    \email{giorgio.poggesi@uwa.edu.au}
\urladdr{https://research-repository.uwa.edu.au/en/persons/giorgio-poggesi-2}


\begin{abstract}
We prove interpolating estimates providing a bound for the oscillation of a function in terms of two $L^p$ norms of its gradient. They are based on a pointwise bound of a function on cones  in terms of the Riesz potential of its gradient.   The estimates hold for a general class of domains, including, e.g., Lipschitz domains. All the constants involved can be explicitly computed.
\par
As an application, we show how to use these estimates to obtain stability for Alexandrov's Soap Bubble Theorem and Serrin's overdetermined boundary value problem. The new approach results in several novelties and benefits for these problems.
\end{abstract}

\keywords{Interpolating estimates, Serrin's overdetermined problem, Alexandrov Soap Bubble Theorem, constant mean curvature, stability, quantitative estimates}
\subjclass[2010]{Primary 35A23, 35N25; Secondary 53A10}

\maketitle

\raggedbottom

\section{Introduction}
Let $\Om \subset \RR^N$ be a bounded domain. For $1 \le p \le \infty$ the
number $\nr f \nr_{p,\Om}$, will denote the $L^p$-norm of a measurable function $f : \Om \to \RR$ with
respect to the normalized Lebesgue measure $d\mu_x = dx/|\Om|$.

In Theorem \ref{thm:p<N and p=N} of the present paper, for $1 \le p \le N$ and $N<q\le\infty$, we prove the following interpolating inequalities, which hold true 
for any $f\in W^{1,q}(\Om)$:
\begin{equation}
\label{inequalities}
\max\limits_{\ol{ \Om}} f - \min\limits_{ \ol{ \Om}} f \le
c\,\begin{cases}
\nr\na f\nr_{p,\Om} &\mbox{ for } \ p>N \\
\nr\na f\nr_{N,\Om} \log\bigl(e\,\nr\na f\nr_{q,\Om}/\nr\na f\nr_{N,\Om}\bigr), \ &\mbox{ for } \ p=N, \\
\nr\na f\nr_{p,\Om}^{\al_{p,q}}\nr\na f\nr_{q,\Om}^{1-\al_{p,q}}, \ &\mbox{ for } \ 1\le p<N.
\end{cases}
\end{equation}
Here,
$$
\al_{p,q}=\frac{p\, (q-N)}{N\,(q-p)} \ \mbox{ for } \ N<q<\infty, \quad \al_{p,\infty}=\frac{p}{N}.
$$

The inequalities in \eqref{inequalities} are valid whenever $\Om$ is a bounded domain satisfying a uniform interior cone condition (see Section \ref{subsec:global estimates} for the definition). 
The constant $c$ appearing in \eqref{inequalities} may be explicitly computed. It only depends on ($N$, $p$, $q$, and) the parameters associated to the uniform  interior  cone condition (see Theorems \ref{thm:case p > N} and \ref{thm:p<N and p=N} for details).
 
The proof of \eqref{inequalities} is based on a pointwise bound on cones  for $f$ in terms of the Riesz potential of its gradient (see Lemma \ref{lem:primo passo}). When $1 \le p \le N$, \eqref{inequalities} is obtained by combining that bound with an interpolation procedure performed on cones (see Lemma~\ref{lem:stima p-q gradiente}). We stress that, for $1 \le p \le N$, \eqref{inequalities} cannot be obtained by simply combining the Sobolev-Morrey embedding with the classical interpolation of $L^p$ spaces.
When $p>N$ instead, \eqref{inequalities} could be directly deduced by the classical Morrey's inequality. However, here we offer an alternative proof, based on the arguments developed in this paper.
\par
As an application, we shall use these inequalities to give an alternative way to obtain, and even improve, certain estimates proved by the authors in \cite[Theorems 2.10 and 2.8]{MP3}. These have been a crucial ingredient to obtain the stability of the spherical configuration for Alexandrov's Soap Bubble Theorem (SBT), Serrin's and other related overdetermined problems (see, e.g., \cite{MP1, MP2, MP3, Po1, PogTesi, DPV, CPY, OO}).
More precisely, \eqref{inequalities} can be used as a substitute of \cite[Lemma 3.14]{PogTesi}
when aiming to obtain those stability results in the spirit of \cite[Theorems 2.10 and 2.8]{MP3}. We shall detail in Section \ref{sec:SBT} how this agenda can be carried out.
We emphasize that, while \cite[Lemma 3.14]{PogTesi} can only be proved for sub-harmonic functions, 
our new bounds
do not need this requirement. Thanks to this feature, they can also be useful in different and more general contexts. More on this will be clarified in forthcoming research. See also the recent paper by Scheuer \cite{Sc}.

In the remainder of this introduction, for the case of the SBT, we briefly describe the main steps of the argument that motivates the application of our interpolating inequalities.
Alexandrov's SBT states that a closed surface $\Ga$, embedded in $\RR^N$, and that has constant \textit{mean curvature} $H$ must be a sphere. Roughly speaking, by stability of the spherical configuration in this problem, we mean an inequality of the type:
$$
\mbox{measure of closeness to a sphere}\le \Psi(\nr H-H_0\nr).
$$
Here, $\Psi$ is a non-negative continuous function vanishing at $0$ and $\nr H-H_0\nr$ is the deviation of  $H$ from a reference constant $H_0$, in a suitable norm. In the literature, there are many different ways to quantify the deviations of $H$ from $H_0$ and of a surface from a sphere (we refer the reader to the works \cite{Ma}, \cite{PogTesi}, and \cite{MP3} for a quite exhaustive list of references).  
It is clear that the weaker the norm $\nr H-H_0\nr$ is and the stronger the distance of $\Ga$ from a sphere is, the better the estimate is. On the other hand, in such a weak-strong setting, it may be difficult to obtain for $\Psi$ the most desirable linear profile: $\Psi(\si)=c\,\si$. Here,  $c$ is some constant depending on some  geometric parametes of the surface, easy to compute if possible. When this occurs, the optimality can be proved by considering sequences of ellipsoids.
\par
In the works \cite{MP1}-\cite{MP3}, we consider the boundary $\Ga$ of a bounded domain $\Om$, we set $H_0$ to be the ratio $|\Ga|/N |\Om|$, and we adopt an $L^2(\Ga)$ (or even $L^1(\Ga)$) deviation of $H$ from $H_0$. Also, we measure the distance of $\Ga$ from a sphere, by the quantity $\rho_e-\rho_i$, where $\rho_i$ and $\rho_e$, $\rho_i\le\rho_e$, are the radii of the best spherical annulus  containing $\Ga$. This will be given by $B_{\rho_e}(z)\setminus\ol{B_{\rho_i}(z)}$ for some $z\in\Om$. In other words, we obtained a  bound of this type:
$$
\rho_e-\rho_i\le c\,\Psi\left(\nr H-H_0\nr_{L^2(\Ga)}\right).
$$
In this setting, in \cite{MP3} we obtained  a linear profile for $\Psi$ in low dimension ($N=2, 3$) and a H\"older profile with exponent $2/(N-2)$, for $N\ge 5$. For the threshold case $N=4$, we got, in a sense,  a profile ``arbitrarily close to a linear one'' (see Remark \ref{rem:old-stability} or \cite{MP3}, for details).

\par
In Section \ref{sec:SBT} of this paper, for surfaces of class $C^2$, we show that the interpolating bounds obtained in Section \ref{sec: interpolating estimates} help to improve the profile for $N\ge 4$. In fact, in Theorem \ref{thm:SBT-improved-stability}, for $N=4$ we improve the older estimate (that was $\Psi(\si)=c_\ve\,\si^{1-\ve}$, for any fixed $\ve>0$) to a sharper and more plausible one: $\Psi(\si)=c\, \si \log(1/\si)$.  Moreover, when $N\ge 5$, we are able to upgrade the profile $\Psi(\si)=c\, \si^{2/(N-2)}$ to $\Psi(\si)=c_\ve\, \si^{4/N-\ve}$, for any fixed $\ve>0$. This profile can be further improved to $\Psi(\si)=c\,\si^{4/N}$, if we consider surfaces of class $C^{2,\ga}$, $1<\ga\le 1$.  For $2\le N \le 3$, we just show that the new bounds in \eqref{inequalities} provide an alternative way to recover the optimal profile previously obtained in \cite{MP2, MP3}.
\par
Another novelty of this paper is that we show that our new improvements also hold if we enforce the quantity $\rho_e-\rho_i$ by replacing it with the stronger deviation:
$$
\rho_e-\rho_i+R\,\left\nr \nu-\frac{\na Q^z}{R}\right\nr_{2,\Ga}
$$
Here, $R=1/H_0$, $\nu$ is the exterior unit normal vector to $\Ga$, and $Q^z$ is defined by
\begin{equation}
\label{quadratic}
Q^z(x)=\frac{|x-z|^2}{2} \ \mbox{ for } \ x, z\in\RR^N.
\end{equation}
(Also in this case, the relevant norm is defined in the the corresponding normalized measure $d S_x/|\Ga|$.)
\par
Thus, the smallness of this new measure of closeness to a sphere  tells us not only that $\Ga$ is uniformly close to a sphere, but also that the Gauss map of $\Ga$ is quantitatively close \textit{in the average} to that of the same sphere. Therefore, all in all, in Theorem \ref{thm:SBT-improved-stability}, we enhance the last up-to-date bounds of \cite{MP3} for the stability of the SBT as follows:
\begin{equation}
\label{new-bound}
\rho_e-\rho_i+R\,\left\nr \nu-\frac{\na Q^z}{R}\right\nr_{2,\Ga}\le c\,\Psi\left(\nr H-H_0\nr_{L^2(\Ga)}\right),
\end{equation}
where 
\begin{equation}
\label{def-Psi-intro}
\Psi(\si)=\begin{cases}
\si \ &\mbox{ if } \ N=2, 3, \\
\si \max\left[\log(1/\si),1\right] \ &\mbox{ if } \ N=4, \\
\si^{\tau}  \ &\mbox{ if } \ N\ge 5,
\end{cases}
\end{equation}
where $\tau=4/N$ if $\Ga$ is of class $C^{2,\ga}$, $1<\ga\le 1$. If $\Ga$ is of class $C^2$, instead, when $N\ge 5$, we obtain that, for any sufficiently small $\ve>0$, there exists a constant $c=c_\ve$  such that \eqref{new-bound}-\eqref{def-Psi-intro} holds with $\tau=4/N-\ve$.
The constant $c$ only depends on $N$, the diameter $d_\Om$ of $\Om$, and parameters associated with the assumed regularity of $\Ga$. If $\Ga$ is of class $C^2$, these are the radii $r_i$ and $r_e$ of the uniform interior and exterior ball condition (see Section \ref{sec:SBT}). If $\Ga$ is of class $C^{2,\ga}$, $c$ depends on a suitable modulus of $C^{2,\ga}$-continuity for $\Ga$.
For details, see Theorem \ref{thm:SBT-improved-stability} and Remark \ref{rem:schauder-calderon-zygmund}. We stress that, for the second summand on the left-hand side of \eqref{new-bound}, we can actually obtain an optimal linear profile of stability, in every dimension (see \eqref{eq:SBT mappa Gauss stability}).
\par
We spend a few final words to explain how the bounds derived in Section \ref{sec: interpolating estimates} come into play to obtain \eqref{new-bound}. To this aim, we let $u\in C^1(\ol{\Om})\cap C^2(\Om)$ be the solution of the problem:
$$
\De u=N \ \mbox{ in } \ \Om, \quad u=0 \ \mbox{ on } \ \Ga.
$$ 
Also, we define the harmonic function $h=u-Q^z$. Notice that, if $z\in\Om$, then we have that
$$
\frac12 \left(\frac{|\Om|}{|B|}\right)^{1/N} (\rho_e-\rho_i)\le \frac12\,(\rho_e^2-\rho_i^2)=
\max_{\Ga}h-\min_{\Ga}h.
$$
Here, $B$ is a unit ball in $\RR^N$. Thus, a bound for the term $\rho_e-\rho_i$ in \eqref{new-bound} descends from the following identity
\begin{equation}
\label{H-fundamental}
\frac1{N-1}\int_{\Om}  |\na ^2 h|^2 dx+
\frac1{R}\,\int_\Ga (u_\nu-R)^2 dS_x = 
\int_{\Ga}(H_0-H)\, (u_\nu)^2 dS_x,
\end{equation}
which was proved in \cite{MP1}. In fact, since the right-hand side can be easily bounded in terms of the $L^2(\Ga)$ norm of $H-H_0$, then the desired bound for $\rho_e-\rho_i$ can be obtained if we can control the oscillation of $h$ on $\Ga$
in terms of the first summand in \eqref{H-fundamental}. This goal is achieved by combining the bounds \eqref{inequalities} (applied to $h$ and its gradient) with some Poincar\'e-type inequality. 
\par
The second summand on the left-hand side of \eqref{new-bound} can instead be estimated by observing that
$$
|R\,\nu-\na Q^z|\le |R\,\nu-u_\nu\,\nu|+|\na u-\na Q^z|= 
|R-u_\nu|+|\na h| \ \mbox{ on } \ \Ga.
$$
The two quantities on the rightest-hand side can be estimated in $L^2(\Ga)$-norm by means of \eqref{H-fundamental} and, again, by some inequalities derived in \cite{MP3}. These involve a trace-type formula, 
$$
\int_\Ga |\na h|^2 dS_x\le c\int_\Om (-u)\, |\na^2 h|^2\,dx,
$$
and another identity (stated in \cite{MP1} and proved in \cite{MP2}):
\begin{equation}
\label{identity-MP}
\int_\Om (-u)\, |\na^2 h|^2\,dx=\frac12\,\int_\Ga (u_\nu^2-R^2)\,h_\nu\,dS_x.
\end{equation}
\par
This last identity, immediately gives radial symmetry for $\Om$ in Serrin's overdetermined problem (that prescribes that $u_\nu$ is constant on $\Ga$).  Together with  the arguments used to obtain \eqref{new-bound}, \eqref{identity-MP} will also help us to upgrade an analogous stability bound for radial symmetry in Serrin's problem.
This task will be accomplished in Theorem \ref{thm:Improved-Serrin-stability}.

\smallskip

\section{Interpolating estimates for Sobolev functions}
\label{sec: interpolating estimates}
Let $\SS^{N-1}$ be the unit sphere in the Euclidean space $\RR^N$, $N\ge 2$. For $\te\in [0,\pi/2]$ and $e\in\SS^{N-1}$, we set
\begin{equation*}\label{def:superficie cono sfera}
\cS_\te = \{\om \in \SS^{N-1} : \cos\te < \lan \om , e \ran \}.
\end{equation*}
This is a spherical cap with axis $e$ and opening width $\te$. We also denote by 
\begin{equation*}
\label{def:cono generica altezza}
\cC_{x, a} = \{ x + a \om : \om \in \cS_\te, \, 0 < s < a \},
\end{equation*}
the finite right spherical cone with vertex at $x$, axis in some direction $e$, and height $a>0$.
In what follows,
$| \cC_{x,a} |$ and $| \cS_\te |$ will denote
indifferently the $N$-dimensional Lebesgue measure of $\cC_{x,a}$ and the $(N-1)$-dimensional surface measure of $\cS_\te$.

\subsection{Pointwise estimates on cones}\label{Subsec:Pointwise estimates on cones}
We start by proving some useful pointwise estimates in cones (see also \cite{Ad}). In what follows, we set $\cC_x=\cC_{x,a}$ and use the normalized Lebesgue measure $d\mu_y=dy/|E|$ for any measurable set $E\subset\RR^N$ of finite measure. 
\begin{lem}
\label{lem:primo passo}
For any $f \in C^1(\ol{\cC}_x)$, it holds that
\begin{equation}
\label{eq:lemma1-1}
\left|f(x)-f_{\cC_x}\right| \le   \int_{\cC_x}\frac{ | \na f ( y )| }{| y - x |^{N-1}} \, \frac{a^N-|y-x|^N}{N}\,d\mu_y.
\end{equation}
In particular, we have that
\begin{equation}
\label{eq:lemma1-2}
\left|f(x)-f_{\cC_x}\right|\le   \frac{a^N}{N }\int_{\cC_{x}} \frac{ | \na f(y) | }{| y - x |^{N-1}} \,d\mu_y.
\end{equation}
\end{lem}
\begin{proof}
By the change of variables $y = x + s\, \om$ for $s\in (0,a)$ and $\om\in\cS_\te$, we write:
\begin{multline*}
f(x)-f_{\cC_x}=f(x)-\int_{\cC_x} f\,d\mu_y=\frac1{|\cC_x|}\int_{\cC_x}[f(x)-f(y)]\,dy= \\
\frac1{|\cC_x|}\int_0^a s^{N-1}\left\{\int_{\cS_\te}[f(x)-f(x+s\,\om)]\,dS_\om\right\} ds,
\end{multline*}
where $dS_\om$ denotes the surface element on $\SS^{N-1}$.
Next, the fundamental theorem of calculus gives:
$$
f(x)-f(x+s\,\om)=-\int_0^s \om\cdot\na f(x+t\,\om)\,dt.
$$
Thus, we can infer that
\begin{multline*}
\left|f(x)-f_{\cC_x}\right|\le \frac1{|\cC_x|}\int_0^a \int_{\cS_\te} s^{N-1} \left[\int_0^s|\na f(x+t\,\om)|\,dt\right]\,dS_\om\,ds= \\
\frac1{|\cC_x|}\int_0^a s^{N-1} \int_{\cS_\te} \left[\int_0^s\frac{|\na f(x+t\,\om)|}{|x+t\,\om-x|^{N-1}}\,t^{N-1}\,dt\right]\,dS_\om\,ds= \\
\frac1{|\cC_x|}\int_0^a s^{N-1} \left[\int_{\cC_{x,s}} \frac{|\na f(y)|}{|y-x|^{N-1}}\,dy \right] ds.
\end{multline*}

Now, by an application of Fubini's theorem we obtain that
$$
\int_0^a\left(\int_{\cC_{x,s}} \frac{ | \na f ( y )| }{| y - x |^{N-1}} \, dy\right) ds=
\int_{\cC_{x}} \frac{ | \na f ( y )| }{| y - x |^{N-1}} \, \frac{a^N-|y-x|^N}{N}\,dy.
$$
Thus, \eqref{eq:lemma1-1} and \eqref{eq:lemma1-2} easily follow.
\end{proof}

As a corollary, we have the following {\it Morrey-Sobolev-type} inequality. The relevant Lebesgue norms are defined with respect to the normalized measure $d\mu_y$.

\begin{cor}
\label{cor:MorreySobolev}
If $N<p \le \infty$ and $f \in C^1(\ol{\cC}_x)$, we have that
\begin{equation}
\label{eq:MorreySobolev}
\left|f(x)-f_{\cC_x}\right| \le  \frac{a}{N}\,\be\left(1-\frac{p'}{N'}, p'+1\right)^{1/p'} \| \na f \|_{p, \cC_x},
\end{equation}
where $\be(\xi,\eta)$ denotes {\rm Euler's beta function}. When $p=\infty$, the inequality is that obtained by taking the limits as $p\to\infty$ of the relevant quantities.
\end{cor}
\begin{proof}
The desired result follows from \eqref{eq:lemma1-1} by applying H\"older's inequality to the right-hand side and the calculation:
\begin{multline*}
\int_{\cC_x} \left(\frac{a^N-|y-x|^N}{|y-x|^{N-1}}\right)^{p'} d\mu_y=
\frac{|\cS_\te|}{|\cC_x|} \int_0^r \left(\frac{a^N-s^N}{s^{N-1}}\right)^{p'} s^{N-1} ds= \\
a^{p'} \int_0^1 (1-t)^{p'} t^{-\frac{N-1}{N} p'} dt=
a^{p'} \be\left(1-\frac{N-1}{N}\,p', p'+1\right).
\end{multline*}
Since $\be(\xi,\eta)$ is well-defined only if $\xi, \eta>0$, we get the restriction $p>N$. As already mentioned, the case $p=\infty$ can be derived by taking the limit as $p\to\infty$.
\end{proof}

\subsection{Global estimates for the oscillation of functions}\label{subsec:global estimates}
Let $\Om \subset \RR^N$ be a bounded domain (i.e., a connected bounded open set) with boundary $\Ga$.
Given $a>0$ and $\te\in [0,\pi/2]$, we say that $\Om$ satisfies the $(\te, a)$-uniform interior cone condition, if for every $x \in \ol{\Om}$ there exists a cone $\cC_{x}$ with opening width $\te$ and height $a$, such that $\cC_{x} \subset\Om$
and
$\ol{\cC}_{x} \cap \Ga = \{ x \}$, whenever $x \in \Ga$.
The following result easily follows from Corollary  \ref{cor:MorreySobolev}.

\begin{cor}
\label{cor:MorreySobolev-Omega}
Let $N<p \le \infty$ and $\Om\subset\RR^N$ be a bounded domain that satsfies the  uniform interior $(\te,a)$-cone property.
For every $x\in\ol{\Om}$ and $f\in W^{1,p}(\Om)$, we have that
\begin{equation}
\label{eq:MorreySobolev-Omega}
\left|f(x)-f_\Om\right|\le k(N,p,\te)\,a^{1-N/p}\, |\Om|^{1/p} \nr\na f\nr_{p,\Om},
\end{equation}
for some constant $k(N,p,\te)$ only depending on $N, p$, and $\te$.
\end{cor}
\begin{proof}
For any $x\in\ol{\Om}$, there is a cone $\cC_x$ contained in $\Om$.
Hence, we apply \eqref{eq:MorreySobolev} to the function $f-f_\Om+f_{\cC_x}$ and infer that
\begin{multline*}
\left|f(x)-f_\Om\right|\le\frac{a}{N}\,\be\left(1-\frac{p'}{N'}, p'+1\right)^{1/p'} \| \na f \|_{p, \cC_x}\le \\
\frac{a}{N}\,\be\left(1-\frac{p'}{N'}, p'+1\right)^{1/p'}\left(\frac{|\Om|}{|\cC_x|}\right)^{1/p} \| \na f \|_{p, \Om}\le \\
\frac{\be\left(1-\frac{p'}{N'}, p'+1\right)^{1/p'}}{N^{1/p'}\, |\cS_\te|^{1/p}} \, a^{1-N/p}\, |\Om|^{1/p} \nr\na f\nr_{p,\Om}.
\end{multline*}
In the second inequality, we use the monotonicity of Lebesgue's integral with respect to set inclusion. 
\end{proof}

In this section, we aim to derive inequalities that bound from above the oscillation on $\ol{\Om}$ of a function $f$ with the $L^p$-norm  of its gradient on $\Om$.

\begin{thm}[The case $p>N$]
\label{thm:case p > N}
Set $N<p\le\infty$. Let $\Om \subset \RR^N$ be a bounded domain satisfying the $(\te, a)$-uniform interior cone condition. 
\par
There exists a constant $k(N,p,\te)$ only depending on $N, p$, and $\te$ such that, for any $f \in W^{1,p}(\Om)$, it holds that
\begin{equation}
\label{eq:p>N theo stima osc grad}
\max_{\ol{\Om}} f - \min_{\ol\Om} f  \le  k(N,p,\te)\, a^{1-N/p}\, |\Om|^{1/p} \,  \nr \na f \nr_{p, \Om}.
\end{equation}
\end{thm}
\begin{proof}
Notice that the oscillation of $f$ at the left-hand side of \eqref{eq:p>N theo stima osc grad} is well defined, since $f$ is continuous on $\ol{\Om}$. 
\par
Let $x_m, x_M\in\ol{\Om}$ be points at which $f$ attains its minimum and maximum. Then, we have that
$$
\max_{\ol{\Om}} f - \min_{\ol\Om} f  \le f(x_M)-f_\Om+f_\Om-f(x_m)
$$
and we conclude by applying twice Corollary \ref{cor:MorreySobolev-Omega}.
\end{proof}

\medskip

It is clear that the proof of Corollary \ref{cor:MorreySobolev} fails when $1\le p \le N$, because of  the singularity at $x$. However,
in this case, we can still obtain a slightly different estimate by means of an interpolation procedure, if information on higher integrability of the gradient of $f$ is available.

\begin{lem}
\label{lem:stima p-q gradiente}
Let $f \in C^1(\ol{\cC}_x)$. Let $1\le p\le N$, $N<q\le\infty$, and set
\begin{equation}\label{eq:def alpha pq}
\al_{p,q}=\frac{p\, (q-N)}{N\,(q-p)}.
\end{equation}
\begin{enumerate}[(i)]
\item
If $1\le p<N$, we have that
\begin{equation}
\label{eq:interpol-p-q-p<N}
a^{N-1} \int_{\cC_x} \frac{ | \na f ( y )| }{| y - x |^{N-1}} \, d\mu_y \le k_{N,p,q}\,\nr\na f\nr_{q,\cC_x}^{1-\al_{p,q}} \, \nr \na f \nr_{p, \cC_x }^{\al_{p,q}} ,
\end{equation}
for some positive constant $k_{N,p,q}$ only depending on $N, p,$ and $q$.

\item
If $p=N$. we have that
\begin{equation}
\label{eq:interpol-p-N}
a^{N-1}\ \int_{\cC_x} \frac{ | \na f ( y )| }{| y - x |^{N-1}} \, d\mu_y \le \frac{q}{q-N}\,\nr \na f \nr_{N, \cC_x } \log\left( \frac{e\,\nr \na f \nr_{q, \cC_x } }{q'\, \nr \na f \nr_{N, \cC_x } }\right).
\end{equation}
\end{enumerate}
\end{lem}

\begin{proof}
For any $\si\in(0,a)$, we compute that
\begin{multline}
\label{eq:interpolation p-q}
\int_{\cC_x} \frac{ | \na f ( y )| }{| y - x |^{N-1}} \, dy = \int_{\cC_{x, \si}} \frac{ | \na f ( y )| }{| y - x |^{N-1}} \, dy + \int_{\cC_x \setminus \cC_{x, \si}} \frac{ | \na f ( y )| }{| y - x |^{N-1}} \, dy \le
\\
\left[\int_{\cC_{x,\si}} \frac{dy}{|y-x|^{q' (N-1)}}\right]^{1/q'} \left(\int_{\cC_{x,\si}} |\na f(y)|^q dy\right)^{1/q}+ \\
 \left[ \int_{\cC_x \setminus \cC_{x, \si} } \frac{dy}{| y - x |^{p'(N-1)}} \right]^{1/p'} \left(\int_{\cC_x\setminus\cC_{x,\si}}|\na f|^p dy\right)^{1/p},
\end{multline}
by H\"older's inequality.
Now,
a direct computation shows that
\begin{multline}
\label{eq:integral-bound-p-q}
\left[\int_{\cC_{x, \si} }\frac{dy}{| y - x |^{q' (N-1)}}\right]^{1/q'} = \left[\frac{q-1}{q-N}\,|\cS_\te|\right]^{1/q'}\si^{\frac{q-N}{q}},  \\
\left[\int_{\cC_x \setminus \cC_{x, \si} } \frac{dy}{| y - x |^{p' (N-1)}}\right]^{1/p'} \!\!\!=
\begin{cases}
\left[\frac{p-1}{N-p} |\cS_\te| \, \left( \si^{-\frac{N-p}{p-1}} - a^{-\frac{N-p}{p-1}}   \right)\right]^{1/p'}   &\text{if $1\le p<  N$}, \vspace{3pt}\\
\left[|\cS_\te| \, \log\frac{a}{\si} \right]^{1/N'}    &\text{if $p = N$}.
\end{cases}
\end{multline}
For $p=1$, this formula must be intended in the limit as $p\to 1$.

\medskip

(i) Let $1\le p<N$. By \eqref{eq:interpolation p-q}, \eqref{eq:integral-bound-p-q}, and some algebraic manipulations, we can infer that
\begin{multline}
\label{eq:intermedia-p-q}
a^{N-1}\int_{\cC_x} \frac{ | \na f ( y )| }{| y - x |^{N-1}}\, d\mu_y  \le  \left[\frac{N\,(q-1)}{q-N}\right]^{1-1/q}\,\nr\na f\nr_{q,\cC_x}\left(\frac{\si}{a}\right)^{1-N/q}+ \\
\left[\frac{N\,(p-1)}{N-p}\right]^{1-1/p}\,\nr\na f\nr_{p,\cC_x}\left(\frac{\si}{a}\right)^{1-N/p}
\end{multline}
for any $\si\in(0,a]$. The minimum of the right-hand side is attained either at
$$
\ol{\si} = a\left[\frac{N (p-1)}{N-p}\right]^{\frac{q (p-1)}{N(q-p)}}  \left[\frac{q-N}{N(q-1)}\right]^{\frac{p (q-1)}{N(q-p)}} \left(\frac{1-\al_{p,q}}{\al_{p,q}}\,\frac{\nr \na f \nr_{p, \cC_x }}{\nr\na f \nr_{q, \cC_x}}\right)^{\frac{p q}{N (q-p)}},
$$
or at $\si=a$.
In the former case, we plug $\ol{\si}$ into \eqref{eq:intermedia-p-q} and obtain \eqref{eq:interpol-p-q-p<N} with some computable constant $k'$.
In the latter case, we have that
$$
\left[\frac{N (p-1)}{N-p}\right]^{\frac{q (p-1)}{N(q-p)}}  \left[\frac{q-N}{N(q-1)}\right]^{\frac{p (q-1)}{N(q-p)}} \left(\frac{1-\al_{p,q}}{\al_{p,q}}\,\frac{\nr \na f \nr_{p, \cC_x }}{\nr\na f \nr_{q, \cC_x}}\right)^{\frac{p q}{N (q-p)}}>1,
$$
since $\ol{\si}>a$. Hence, by means of this inequality and the fact that we have that
$$
\int_{\cC_x} \frac{ | \na f ( y )| }{| y - x |^{N-1}} d\mu_y  \le \frac1{a^{N-1}}\left[\frac{N (q-1)}{q-N}\right]^{1-1/q}\,\nr\na f \nr_{q, \cC_x} \left(\frac{\si}{a}\right)^{1-N/q},
$$
thanks to \eqref{eq:interpolation p-q},
we again obtain \eqref{eq:interpol-p-q-p<N} for some possibly different computable constant $k''$. Thus, 
we conclude that \eqref{eq:interpol-p-q-p<N} holds true with $k_{N,p,q}=\max(k', k'')$.

\medskip

(ii) Let $p=N$. We proceed as in the case (i) by putting together \eqref{eq:interpolation p-q} and \eqref{eq:integral-bound-p-q}. After some calculation, we obtain:
\begin{multline*}
a^{N-1}\int_{\cC_x} \frac{ | \na f ( y )| }{| y - x |^{N-1}} \, d\mu_y  \le \\
\left[\frac{N (q-1)}{q-N}\right]^{1-1/q}\nr\na f\nr_{q,\cC_x} \left(\frac{\si}{a}\right)^{1-N/q}+\left(N\,\log\frac{a}{\si}\right)^{1-1/N} \nr\na f\nr_{N,\cC_x}.
\end{multline*}
If we assume that $0<\si < a / e$, being as $N (q-1)\ge q-N$, we can simplify this inequality to get that
\begin{equation}
\label{eq:pre-inerp-N}
\frac{a^{N-1}}{N} \int_{\cC_x} \frac{ | \na f ( y )| }{| y - x |^{N-1}} \, d\mu_y  \le 
\frac{q-1}{q-N}\,\nr\na f\nr_{q,\cC_x} \left(\frac{\si}{a}\right)^{1-N/q}+\ \nr\na f\nr_{N,\cC_x} \log\frac{a}{\si}.
\end{equation}
Thus, the minimum of the right-hand side is attained either at
$$
\si=\ol{\si}=a\,\left[\frac{q'\,\nr\na f\nr_{N,\cC_x}}{\nr\na f\nr_{q,\cC_x}}\right]^{\frac{q}{q-N}} \ \mbox{ or at } \ \si=a/e.
$$
\par
In the former case, we get that
$$
\frac{a^{N-1}}{N}\,\int_{\cC_x} \frac{ | \na f ( y )| }{| y - x |^{N-1}} \, d\mu_y  \le
\frac{q}{q-N} \,\nr \na f \nr_{ N , \cC_x } \log\left(\frac{e\, \nr \na f \nr_{q, \cC_x }}{ q'\,\nr \na f \nr_{N, \cC_x } }\right),
$$
that is \eqref{eq:interpol-p-N} holds true.
In the latter case, we have that
$$
e^{-1}\le \left[\frac{q'\,\nr\na f\nr_{N,\cC_x}}{\nr\na f\nr_{q,\cC_x}}\right]^{\frac{q}{q-N}},
$$
since $\ol{\si}\ge a/e$. Thus, we get that
\begin{multline*}
\frac{a^{N-1}}{N}\,\int_{\cC_x} \frac{ | \na f ( y )| }{| y - x |^{N-1}} \, d\mu_y \le
\frac{q-1}{q-N}\nr\na f\nr_{q,\cC_x} e^{N/q-1}\le \\
\frac{q}{q-N}\,\nr\na f\nr_{N,\cC_x} \le
\frac{q}{q-N}\,\nr \na f \nr_{N, \cC_x }\,\log\left(\frac{e\,\nr\na f\nr_{q,\cC_x} }{q' \nr \na f \nr_{N, \cC_x } }\right),
\end{multline*}
being as $\nr\na f\nr_{N,\cC_x}\le \nr\na f\nr_{q,\cC_x}$. 
\end{proof}
As done for Corollary \ref{cor:MorreySobolev-Omega}, we obtain following consequence.

\begin{cor}
\label{cor:bound-omega}
For any cone $\cC_x\subset\Om$ of height $a$ and opening width $\te$, it holds that
$$
\int_{\cC_x} \frac{ | \na f ( y )| }{| y - x |^{N-1}} \, d\mu_y \le k(N,p,q, \te)\,\frac{|\Om|}{a^{2N-1}}
\,\nr\na f\nr_{q,\Om}^{1-\al_{p,q}} \, \nr \na f \nr_{p, \Om }^{\al_{p,q}},
$$
for $1\le p<N$ and, if $p=N$,
\begin{multline*}
\int_{\cC_x} \frac{ | \na f ( y )| }{| y - x |^{N-1}} \, d\mu_y \le 
k(N,p,q,\te)\,\frac{|\Om|^{1/N}}{a^N}\,\nr \na f \nr_{N, \Om } \log\left(
\frac{|\Om|^{1/q-1/N}}{a^{N/q-1}}\frac{\nr \na f \nr_{q, \Om} }{\nr \na f \nr_{N, \Om} }\right),
\end{multline*}
for some constant $k(N,p,q,\te)$ only depending on $N, p, q$, and $\te$.
\end{cor}
\begin{proof}
The monotonicity of Lebesgue's measure with respect to set inclusion and \eqref{eq:pre-inerp-N} easily give:
\begin{multline*}
\frac{a^{N-1}}{N} \int_{\cC_x} \frac{ | \na f ( y )| }{| y - x |^{N-1}} \, d\mu_y  \le \\
\frac{q-1}{q-N}\left(\frac{|\Om|}{|\cC_x|}\right)^{1/q}\nr\na f\nr_{q,\Om} \left(\frac{\si}{a}\right)^{1-N/q}+\left(\frac{|\Om|}{|\cC_x|}\right)^{1/N}\ \nr\na f\nr_{N,\Om} \log\frac{a}{\si}.
\end{multline*}
Thus, we can proceed as in the last part of the proof of Lemma \ref{lem:stima p-q gradiente}, with similar algebraic manipulations.
\end{proof}
\par
In light of Corollary  \ref{cor:bound-omega}, we can somewhat extend the bound \eqref{eq:p>N theo stima osc grad} to
the case $1\le p\le N$, provided $f\in W^{1,q}(\Om)$ for $q>N$. The proof is straightforward and runs as that of Theorem \ref{thm:case p > N}.

\begin{thm}
\label{thm:p<N and p=N}
Let $1\le p\le N $, $N<q\le\infty$, and set $\al_{p,q}$ as in \eqref{eq:def alpha pq}. 
Let $\Om \subset \RR^N$ be a bounded domain satisfying the $(\te, a)$-uniform interior cone condition. 
\par
For any $f\in W^{1,q}(\Om)$, it holds that
\begin{equation*}
	\max\limits_{\ol{ \Om}} f - \min\limits_{ \ol{ \Om}} f \le 
	k(N, p, q,\te)\,  \frac{|\Om|^{1/p}}{a^{N/p-1}} \,
	\nr\na f\nr_{p,\Om}^{\al_{p,q}}\nr\na f\nr_{q,\Om}^{1-\al_{p,q}},
\end{equation*}
if $1 \le p<N$ and, if $p=N$,
\begin{equation*}
	\max\limits_{\ol{ \Om}} f - \min\limits_{ \ol{ \Om}} f \le 
	k(N,p,q,\te)\,\frac{|\Om|^{1/N}}{a^N}\,\nr \na f \nr_{N, \Om } \log\left(
\frac{|\Om|^{1/q-1/N}}{a^{N/q-1}}\frac{\nr \na f \nr_{q, \Om} }{\nr \na f \nr_{N, \Om} }\right).
\end{equation*}	
Here,  $k(N, p, q,\te)$ is some constant only depending on $N, p, q$, $\te$.
\end{thm}

\section{Application to quantitative symmetry \\  for the Soap Bubble Theorem}
\label{sec:SBT}

As already mentioned, Theorems \ref{thm:case p > N} and \ref{thm:p<N and p=N} give an alternative way to obtain, and even up-grade, the bounds in \cite[Theorems 2.10 and 2.8]{MP3}.
As a by-product, we also obtain new up-graded versions
of stability estimates for the Soap Bubble Theorem
and Serrin's symmetry result.
In this and the next section, we shall give some details on how to obtain the new versions of those stability results. Of course, a similar reasoning can be applied to other stability results contained
in \cite{MP1,MP3,PogTesi,DPV}.

\subsection{Preliminary notations and useful bounds}

For a point $z\in\Om$, $\rho_i$ and $\rho_e$ shall denote 
the radius of the largest ball contained in $\Om$
and that of the smallest ball that contains $\Om$, both centered at $z$; in formulas, 
\begin{equation}
\label{def-rhos}
\rho_i=\min_{x\in\Ga}|x-z|  \ \mbox{ and } \ \rho_e=\max_{x\in\Ga}|x-z|.
\end{equation}
\par
We say that $\Om$ satisfies a \textit{uniform interior sphere condition} (with radius $r$) if
for every $p \in \Ga$ there exists a ball $B_r\subset\Om$ such that $\pa B_r\cap\Ga=\{ p\}$; $\Om$ satisfies a \textit{uniform exterior sphere condition} if $\RR^N \setminus \ol{\Om}$ satisfies a uniform interior sphere condition. From now on, we will consider a bounded domain $\Om$ with boundary $\Ga$ of  class $C^{2}$, so that $\Om$ satisfies both a uniform interior and exterior sphere condition. We shall denote by $r_i$ and $r_e$ the relevant respective radii.
It is trivial to check that when $\Om$ satisfies the interior condition with radius $r_i$, then it satisfies the uniform interior $(\te, a)$-cone condition with
\begin{equation}\label{eq:sferainterna allora conointerno}
\te= \frac{ \sqrt{2} }{ 2 } ,  \quad a = r_i.
\end{equation}

\medskip

Next, we consider the solution $u\in C^0(\ol{\Om})\cap C^2(\Om)$ of
\begin{equation}
\label{serrin1}
\De u=N \ \mbox{ in } \ \Om, \quad u=0 \ \mbox{ on } \ \Ga.
\end{equation}
It is well-known that $u\in C^{m,\ga}(\ol{\Om})$ if $\Ga$ is of class $C^{m,\ga}$, $0<\ga\le 1$, for $m=1, 2,\cdots$.

By $M$ we denote a uniform upper bound for the gradient of $u$ on $\ol{\Om}$, in formulas, 
\begin{equation*}
	\label{bound-gradient}
	M \ge \max_{\ol{\Om}} |\na u|=\max_{\Ga} u_\nu.
\end{equation*}
As shown in \cite[Theorem 3.10]{MP1}, 
we can choose an explicit value for $M$:
\begin{equation}
	\label{bound-M}
	M = (N+1)\,\frac{d_\Om(d_\Om+r_e)}{2 r_e}.
\end{equation}

By following \cite{MP3,PogTesi}, we consider the harmonic function
$$
h=Q^z-u,
$$
where $Q^z$ is defined in \eqref{quadratic}.
Notice that, if $z \in \Om$, it holds that
\begin{equation}\label{oscillation}
\max_{\Ga} h-\min_{\Ga} h=\frac12\,(\rho_e^2-\rho_i^2)\ge \left(\frac{|\Om|}{|B|}\right)^{1/N} \frac{\rho_e-\rho_i}{2}\ge \frac{r_i}{2} \, (\rho_e-\rho_i).
\end{equation}
The left-hand side of this inequality can be estimated by Theorems \ref{thm:case p > N} and \ref{thm:p<N and p=N}.
\par
As in \cite{MP3}, it will be convenient to choose $z\in\Om$ as a global minimum point of $u$.
We know from \cite{MP5} that, in this case, the distance $\de_\Ga(z)$ of $z$ to $\Ga$ can be estimated from below in terms of the \textit{inradius} $r_\Om$ (the radius of a maximal ball contained in $\Om$). In fact, in light of \cite[Theorem 1.1]{MP5}, it holds that
\begin{equation}
\label{bound-distance-1}
\de_\Ga(z)\ge \frac{r_\Om}{\sqrt{N}},
\end{equation}
if $\Om$ is mean convex. If $\Ga$ is a general surface of class $C^2$, \cite[Corollary 2.7]{MP5} gives instead the slightly poorer bound:
\begin{equation}
\label{bound-distance-2}
\de_\Ga(z)\ge \frac{r_\Om}{\sqrt{N}} \left[1+\frac{N^2-1}{2N}\,\frac{d_\Om}{r_e}\left(1+\frac{d_\Om}{r_e}\right)\right]^{-1/2}.
\end{equation}

\begin{rem}[On the normalized norms]
\label{rem:normalizations}
{\rm
For the sake of consistency with the previous sections, we will continue to denote by $\nr \cdot \nr_{p , \Om}$ and $\nr \cdot \nr_{p , \Ga}$ the $L^p$-norms in the relevant normalized  measure. Since it holds that 
\begin{equation*}
\label{eq:volume monotonicity}
|B| \, r_\Om^N \le |\Om| \le |B|\,d_\Om^N \quad \mbox{ and } \quad 
N\,|B|\,r_\Om^{N-1}	\le | \Ga | \le N\, \frac{ |\Om| }{r_i},
\end{equation*}
such norms are equivalent to the standard ones. The first three inequalities follow from the inclusions $B_{r_\Om}\subset\Om\subset B_{d_\Om}$.
The last inequality is obtained by putting together the identity
$$
N|\Om|= \int_{\Ga} u_\nu \, dS_x
$$
with the inequality $u_\nu \ge r_i$, which holds true at any point in $\Ga$, by an adaptation of Hopf's lemma (see \cite[Theorem 3.10]{MP1}). 
\par
Notice that, since $r_\Om\ge r_i$, $r_\Om$ can be replaced by $r_i$ in all the relevant formulas.
}
\end{rem}

In what follows, we use the letter $c$ to denote a constant whose value may change line by line.  The dependence of $c$ on
the relevant parameters will be indicated whenever it is important.
All the constants $c$ can be explicitly computed (by following the steps in the relevant proofs) and estimated in terms of the indicated parameters only.	

\subsection{Bounds for $\rho_e-\rho_i$ in terms of $h$}
\label{subsec:bounds-osc-rho}
By applying Theorems \ref{thm:case p > N} and \ref{thm:p<N and p=N} to $h$, we easily obtain the starting point of our analysis.
\begin{lem}
\label{lem:intermediate h}
Let $\Om\subset\RR^N$, $N\ge 2$, be a bounded domain with boundary $\Ga$ of class $C^{2}$. Let
$z$ be a point in $\Om$, and
consider the function  $h=Q^z-u$, with $Q^z$ defined in \eqref{quadratic}. 
\par
There exists a constant $c=c(N, p, r_i)$ such that
\begin{equation*}
\label{eq:stima-grad}
\rho_e - \rho_i  \le 
c\,\begin{cases}
\nr \na h \nr_{p, \Om}  \ &\mbox{if $p>N$}; \\
\displaystyle \nr \na h \nr_{N, \Om} \log \left( \frac{ e \, \nr \na h \nr_{\infty, \Om}  }{ \nr \na h \nr_{N, \Om} } \right)   \ &\mbox{if $p=N$}; \\
\nr \na h \nr_{\infty, \Om}^{ (N-p)/N } \nr \na h \nr_{p, \Om}^{p/N}   \ &\mbox{if $1\le p<N$.}
\end{cases}
\end{equation*}
\end{lem}
\begin{proof}
We apply Theorems \ref{thm:case p > N} and \ref{thm:p<N and p=N}, with $f=h$ and $q=\infty$. By taking into account \eqref{oscillation} and \eqref{eq:sferainterna allora conointerno}, the desired estimates 
easily follow.
(Notice that
\eqref{eq:sferainterna allora conointerno} informs us that in Theorems \ref{thm:case p > N} and \ref{thm:p<N and p=N} we can take $a=r_i$.)
\end{proof}

\begin{rem}[Weighted Poincar\'e inequality]
\label{rem:weighted-grad}
{\rm
Here, we recall a bound for the gradient of $h$, which we will need
in the sequel. Since $z$ is a critical point of $h$ (being as $\na h(z)=\na Q^z(z)-\na u(z)=0$), we know from \cite[Corollary 2.3]{MP3} that
$h$ satisfies the weighted Poincar\'e inequality
\begin{equation*}
	\nr \na h \nr_{r, \Om} \le c\, \nr\de_\Ga^\al\na^2 h  \nr_{p, \Om}.
\end{equation*}
Here, $r, p, \al$ are three numbers such that
\begin{equation*}
	1 \le p \le r \le \frac{N\,p}{N-p\,(1 - \al )} , \quad p\,(1 - \al)<N , \quad 0 \le \al \le 1.
\end{equation*}
The constant $c$ can be explicitly computed by putting together item (iii) of \cite[Remark 2.4]{MP3}, \eqref{bound-distance-2},  and the normalizations discussed in Remark \ref{rem:normalizations}. In detail, we can compute that
$$
c \le k_{N,r,p,\al} \,  |\Om|^{\frac{1-\al}{N}} (d_\Om/r_i)^N\left[N+(N^2-1) \frac{d_\Om}{2 r_e}\left(1+\frac{d_\Om}{r_e}\right)\right]^{N/2} ,
$$
for some constant $k_{N,r,p,\al}$ only depending on $N, r, p,\al$. When $\Ga$ is mean convex, the term in square brackets in the bound above can be removed, by using \eqref{bound-distance-1} in place of \eqref{bound-distance-2}.

}
\end{rem}

As described in the introduction, in order to obtain stability estimates for the Soap Bubble Theorem, we must associate the difference $\rho_e - \rho_i$ with the $L^2$-norm of the hessian matrix $\na^2 h$. The following result gives this association.

\begin{thm}
\label{thm:SBT-W22-stability}
Let $\Om\subset\RR^N$, $N\ge 2$, be a bounded domain with boundary $\Ga$ of class $C^{2}$. Let $z \in \Om$ be a global minimum point of $u$ in $\ol{\Om}$ and set $h=Q^z-u$.
Then, there exists a constant $c=c(N, r_i, r_e, d_\Om)$ such that
$$
\label{ineq:diff-radii-hessian}
\rho_e - \rho_i \le  
c\,\begin{cases}
\nr \na^2 h \nr_{2, \Om}  \ &\mbox{for $N=2, 3$}; \vspace{3pt} \\
\displaystyle \nr \na^2 h \nr_{2,\Om}  \max \left[ \log \left(  \frac{  e \, \nr \na h \nr_{\infty, \Om} }{  \nr \na^2 h \nr_{2,\Om} } \right) , 1\right], \ &\mbox{for $N=4$}; \vspace{3pt}\\
\nr \na h \nr_{\infty, \Om}^{ \frac{N-4}{N-2} }\,  \nr \na^2 h \nr_{2,\Om}^{ \frac{2}{N-2}} ,   \ &\mbox{for $N\ge 5$.}
\end{cases}
$$
\end{thm}

\begin{proof}
(i) Lemma \ref{lem:intermediate h}  with $p = 6$ gives that
\begin{equation*}\label{eq:dimSBTW22N23-step 1}
\rho_e - \rho_i  \le 
c
\, \nr \na h \nr_{6, \Om}\le c\, \nr \na^2 h  \nr_{2, \Om}.
\end{equation*}
The last inequality follows from Remark \ref{rem:weighted-grad} with $r=6$, $p=3/2$, and $\al= 0$, and H\"older's inequality, for $N=2$, and directly from Remark \ref{rem:weighted-grad} with $r=6$, $p=2$, and $\al= 0$, for $N=3$.
\par
(ii) Let $N=4$.
We use Lemma \ref{lem:intermediate h} with $p=N=4$ and get:
$$
\rho_e - \rho_i  \le 
c \,\max \left\lbrace \nr \na^2 h \nr_{4,\Om}  \log \left(  \frac{  e \, \nr \na h \nr_{\infty, \Om} }{  \nr \na^2 h \nr_{4,\Om} } \right) , \nr \na^2 h \nr_{4,\Om} \right\rbrace.   
$$
Next, Remark \ref{rem:weighted-grad} with $r=4$, $p=2$, $\al= 0$, gives:
$$
\nr \na h \nr_{4 , \Om} \le c  \, \nr \na^2 h \nr_{2,\Om}.
$$
Thus, the desired conclusion ensues by invoking the monotonicity of the function 
$t\mapsto t \max \{\log( A/t), 1\}$ for every $A>0$.

(iii) When $N\ge 5$, we can use Lemma \ref{lem:intermediate h} with $p=2N/(N-2)$ and  put it together with
Remark \ref{rem:weighted-grad} with $r=2N/(N-2)$, $p=2$, and $\al= 0$. 
\end{proof}
 
\begin{rem}
\label{rem:grad-h-hessian-h}
{\rm
For $N\ge 4$ the estimates of this theorem depend on $\nr \na h \nr_{\infty, \Om}$. Thus, as done in \cite{MP3}, since we know that
$$
\nr\na h\nr_{\infty,\Om}\le M+d_\Om,
$$
we can easily bound $\rho_e-\rho_i$ in terms of some constant (which possibly depends on $r_i, r_e$, and $d_\Om$, thanks to \eqref{bound-M}) and the number $\nr \na^2 h\nr_{2,\Om}$. Thanks to identity \eqref{H-fundamental}, this number is connected to the deviation $H-H_0$. This will lead to the asymptotic profile in the quantitative symmetry estimate for the Soap Bubble Theorem obtained in \cite{MP3}, with an improvement for the case $N=4$.
\par
However, notice that, when $\Om$ is near a ball in some good topology, the function $h$ tends to be a constant, and hence $\nr \na h \nr_{\infty, \Om}$ tends to be zero. Thus, we expect to improve the relevant bounds in Theorem \ref {thm:SBT-W22-stability}, once we can control
$\nr \na h \nr_{\infty, \Om}$ in terms of $\nr\na^2 h\nr_{2,\Om}$. This control will in turn benefit the quantitative symmetry estimate we are aiming to.  It turns out that an adaptation of our Theorem \ref{thm:p<N and p=N} gives
such desired bound for $\nr \na h \nr_{\infty, \Om}$, if an a priori bound for $\nr \na^2 h \nr_{q,\Om}$ for large $q$ is available, as the following corollary states.
}
\end{rem}

\begin{cor}
\label{cor:bound-for-gradient}
Let $\Om\subset\RR^N$ be a bounded domain with boundary of class $C^2$. Let $1\le p<N$,
$N<q\le\infty$, and set $\al_{p,q}$ as in \eqref{eq:def alpha pq}. Then, if $h\in W^{2,q}(\Om)$, it holds that
\begin{equation}\label{eq:new gradient bound for h}
\nr \na h \nr_{\infty, \Om} \le
c\,  \frac{|\Om|^{1/p}}{r_i^{N/p-1}} \,
	\nr\na^2 h\nr_{p,\Om}^{\al_{p,q}}\nr\na^2 h\nr_{q,\Om}^{1-\al_{p,q}}.
\end{equation}
Here, $c$ is a constant only depending on $N$, $p$, $q$.
\end{cor}

\begin{proof}
Since $\Ga$ is of class $C^2$ , $\Om$ has the uniform interior cone property with $\te=\sqrt{2}/2$ and $a=r_i$.  Let $x\in\ol{\Om}$
%
%
and let $\ell$ be any unit vector.
Applying Theorem \ref{thm:p<N and p=N} and using that, with our choice of $z$, $|h_\ell(x)|=|h_\ell(x)-h_\ell(z)|$, we have that
\begin{multline*}
|h_\ell(x)|\le 
	k(N, p, q)\,  \frac{|\Om|^{1/p}}{r_i^{N/p-1}} \,
	\nr\na h_\ell\nr_{p,\Om}^{\al_{p,q}}\nr\na h_\ell\nr_{q,\Om}^{1-\al_{p,q}}\le \\
k(N, p, q)\,  \frac{|\Om|^{1/p}}{r_i^{N/p-1}} \,
	\nr\na^2 h\nr_{p,\Om}^{\al_{p,q}}\nr\na^2 h\nr_{q,\Om}^{1-\al_{p,q}},
\end{multline*}
where we used the pointwise inequality
$|\na h_\ell|\le|\na^2 h|$. Hence, taking the supremum over all directions $\ell$ yields the desired conclusion.
\par
An inspection of the proof tells us that the corollary could be stated for a domain satisfying an interior cone condition.
 \end{proof}

This corollary allows us to upgrade Theorem \ref{thm:SBT-W22-stability} for $N\ge 5$. Notice that, for $N=4$, we would not get any subtantial improvement, due to the presence of the logarithm in the relevant claim of that theorem.

\begin{cor}
\label{cor:SBT-W22-stability}
Let $\Om\subset\RR^N$, $N\ge 5$, be a bounded domain with boundary $\Ga$ of class $C^2$. Let $z \in \Om$ be a global minimum point of $u$ in $\ol{\Om}$, set $h=Q^z-u$, and suppose that $h\in W^{2,q}(\Om)$.
Then, for every $q\in(N,\infty]$, there exists a constant $c=c(N, q, r_i, r_e, d_\Om)$ such that
$$
\rho_e - \rho_i \le  
c\,\nr \na^2 h \nr_{q, \Om}^{ \frac{q (N-4)}{ (q-2) N} }  \nr \na^2 h \nr_{2,\Om}^{\frac{4}{N}-\frac{2(N-4)}{N (q-2)} }.
$$
\end{cor}

\begin{proof}
Our claim simply follows by combining Theorem \ref{thm:SBT-W22-stability} and Corollary \ref{cor:bound-for-gradient} with the choice $p=2$.
\end{proof}

\subsection{Quantitative symmetry results}

We are now in position to obtain our new quantitative estimates of radial symmetry per the Soap Bubble Theorem. As already mentioned, all we have to do is to relate the norm $\nr \na^2 h \nr_{2,\Om}$ to the deviation of $H$ from $H_0$ in some norm. 
\par
The quantities $\nr \na h \nr_{\infty,\Om}$ and  $\nr \na^2 h \nr_{q,\Om}$ in Theorem \ref{thm:SBT-W22-stability} and Corollary \ref{cor:SBT-W22-stability} will contribute to the computation of the constant in the desired stability profile, as explained in the next remark.

\begin{rem}
\label{rem:schauder-calderon-zygmund}
{\rm
We shall consider two regularity assumptions on $\Ga$.
\par
(i) When $\Ga$ is of class $C^2$, we have that $u\in W^{2,q}(\Om)$ for any $q\in [1,\infty)$ and an a priori bound for $\nr \na^2 h \nr_{q,\Om}$ can be obtained, by the standard $L^q$ estimates for elliptic equations,  being as $\na^2 h = I - \na^2 u$. In fact, by putting together \cite[Theorems 914 and 9.15]{GT}, even under the weaker assumption of $\Ga\in C^{1,1}$, we can obtain for $u$ the bound
$$
\nr \na^2 u\nr_{q,\Om}\le C \ \mbox{ for } \ N<q<\infty,
$$
where $C$ only depends on $N, q$, $|\Om|$, and the regularity $\Om$ (and may blow up as $q\to\infty$). It is well known that $\Ga$ is of class $C^{1,1}$ if and only if it satisfies both the interior and exterior ball condition. Thus, we can claim that $C$ only depends on $N, q, d_\Om, r_i$, and $r_e$.

\par
(ii) When $\Ga$ is of class $C^{2,\ga}$ with $0<\ga\le 1$, we can obtain an a priori bound also for $\nr \na^2 h \nr_{\infty,\Om}$, by standard Schauder's estimates for $\na^2 u$ (see \cite{GT}), in terms of the $C^{2, \ga}$-modulus of continuity $\om_{2,\ga}$ of $\Ga$. (For a definition of $\om_{2,\ga}$, see e.g. \cite[Remark 1]{ABR}.)
}
\end{rem}

The following theorem clearly gives \eqref{new-bound}.

\begin{thm}[Soap Bubble Theorem: enhanced stability] 
\label{thm:SBT-improved-stability}
Let $N\ge 2$ and let $\Om\subset\RR^N$ be a bounded domain with boundary $\Ga$ of class $C^{2}$. Denote by $H$ the mean curvature of $\Ga$ and set $R=N | \Om |/|\Ga|$ and $H_0=1/R$.
\par
Let $z \in \Om$ be a global minimum point of the solution $u$ of \eqref{serrin1} and let $\rho_i$ and $\rho_e$ be defined by \eqref{def-rhos}.
Then, the following inequalities hold true.
\begin{enumerate}[(i)]
\item 
If $2\le N\le 4$, there exists a constant $c = c( N, d_\Om, r_i, r_e )$ such that
\begin{equation}
	\label{eq:final-estimate SBT for C2 rhoei}
	\rho_e - \rho_i  \le 
	c\,\begin{cases}
		\nr H_0 - H \nr_{2, \Ga},  \ &\mbox{if $N=2, 3$}, \\
\nr H_0 - H \nr_{2, \Ga} \max \left[ \log \left( \frac{1  }{ \nr H_0 - H \nr_{2, \Ga} } \right) , 1 \right],   \ &\mbox{if $N=4$}.
	\end{cases}
\end{equation}
\item
 If $N\ge 5$, for any $q\in(N,\infty)$,  there exists a constant $c = c( N,q, d_\Om, r_i, r_e )$ such that
\begin{equation}
\label{stability-SBT-C2-N-5}
\rho_e - \rho_i  \le c\,
		\nr H_0 - H \nr_{2, \Ga}^{\frac{4}{N}-\frac{2(N-4)}{N (q-2)} }.
\end{equation}
\end{enumerate}
Moreover, (for any $N\ge 2$) we have that
\begin{equation}\label{eq:SBT mappa Gauss stability}
R\,\left\nr \nu-\frac{\na Q^z}{R}\right\nr_{2,\Ga} \le c\, \nr H_0-H\nr_{2,\Ga} .
\end{equation}

If $\Ga$ is of class $C^{2,\ga}$, $0<\ga\le 1$, the exponent  in \eqref{stability-SBT-C2-N-5} can be replaced by its limit as $q\to\infty$, i.e. $4/N$.
In this case, the relevant constant $c$ only depends on $N$, $d_\Om$, and the $C^{2,\ga}$-modulus of continuity of $\Ga$.
\end{thm}
\begin{proof}

Inequalities \eqref{eq:final-estimate SBT for C2 rhoei} and \eqref{stability-SBT-C2-N-5} will simply follow from the inequality:
\begin{equation}
\label{eq:inproofSBTpassaggioquasifinal}
\nr \na ^2 h\nr_{2,\Om}\le c\,\nr H-H_0\nr_{2,\Ga}.
\end{equation}
This was proved in \cite{MP3}. 
\par
For the reader's convenience,  we summarize the main steps in the proof of \cite[Theorem 3.5]{MP3}, which lead to \eqref{eq:inproofSBTpassaggioquasifinal}, with the necessary modifications. As usual, the constant $c$ may change from line to line and only depends on quantities (e.g., $R$, $\nr u_\nu \nr_{\infty, \Ga}$, $\nr Q^z_\nu \nr_{\infty, \Ga}$) that, in turn, can be bounded in terms of the parameters indicated in the statement.
\par
The starting point is a modification of the fundamental identity \eqref{H-fundamental}:
\begin{multline*}
\label{identity-SBT-h}
\frac1{N-1}\int_{\Om} |\na ^2 h|^2 dx+
\frac1{R}\,\int_\Ga (u_\nu-R)^2 dS_x = \\
-\int_{\Ga}(H_0-H)\,h_\nu\,u_\nu\,dS_x+
\int_{\Ga}(H_0-H)\, (u_\nu-R)\,Q^z_\nu\, dS_x.
\end{multline*}
\par
Next, if we discard the first summand in this identity, by Cauchy-Schwarz inequality we obtain that
\begin{equation}
\label{ineq:deviation-u-nu}
\nr u_\nu -R\nr_{2,\Ga}^2\le c\,\nr H-H_0\nr_{2,\Ga} \bigl(\nr h_\nu\nr_{2,\Ga}+\nr u_\nu -R\nr_{2,\Ga}\bigr).
\end{equation}
Instead, if we discard the second summand, we can infer that
\begin{equation}
\label{ineq:hessian-h}
\int_{\Om} |\na ^2 h|^2 dx\le c\,\nr H-H_0\nr_{2,\Ga} \bigl(\nr h_\nu\nr_{2,\Ga}+\nr u_\nu -R\nr_{2,\Ga}\bigr).
\end{equation}
\par
Now, we use the fact that we can control $\na h$ (and hence $h_\nu$) on $\Ga$ in terms of the deviation $u_\nu-R$. This is obtained by combining a trace-type inequality for $h$ derived in \cite[Lemma 2.5]{MP3} and identity \eqref{identity-MP}, as follows:
\begin{multline*}
\int_\Ga|\na h|^2 dS_x\le c\,\int_\Om (-u)\, |\na^2 h|^2\,dx=\frac12\,c\,\int_\Ga (u_\nu^2-R^2)\,h_\nu\,dS_x\le \\
c\,\nr u_\nu-R\nr_{2,\Ga} \nr h_\nu\nr_{2,\Ga}\le c\,\nr u_\nu-R\nr_{2,\Ga}\, \nr \na h\nr_{2,\Ga}.
\end{multline*}
 This then gives:
\begin{equation}
\label{trace-grad-h-serrin-deviation}
\nr h_\nu\nr_{2,\Ga}\le\nr \na h\nr_{2,\Ga}\le c\,\nr u_\nu -R\nr_{2,\Ga}.
\end{equation}
\par
Thus, inserting this inequality into \eqref{ineq:deviation-u-nu} gives that 
\begin{equation}
\label{bound-u-nu-curvature}
\nr u_\nu -R\nr_{2,\Ga}\le c\,\nr H-H_0\nr_{2,\Ga}.
\end{equation}
Also, by plugging it into \eqref{ineq:hessian-h}, we infer that
$$
\int_{\Om} |\na ^2 h|^2 dx\le c\,\nr H-H_0\nr_{2,\Ga} \nr u_\nu -R\nr_{2,\Ga}\le c\,\nr H-H_0\nr_{2,\Ga}^2.
$$
Therefore, 
\eqref{eq:inproofSBTpassaggioquasifinal}
follows at once.

\medskip

Now, we proceed to prove \eqref{eq:final-estimate SBT for C2 rhoei} and \eqref{stability-SBT-C2-N-5}. The cases $N=2, 3$ easily follow from Theorem \ref{thm:SBT-W22-stability}. Thus, we are left to prove it  for $N\ge 4$.
\par
For $N=4$, we simply combine Theorem \ref{thm:SBT-W22-stability} and the first part of Remark \ref{rem:grad-h-hessian-h}. Indeed, $\nr\na h\nr_{\infty,\Om}$ is bounded by a constant which only depends on $r_i, r_e$, and $d_\Om$.
\par
For $N\ge 5$,
instead, we use Corollary \ref{cor:SBT-W22-stability} and Remark \ref{rem:schauder-calderon-zygmund}, which give
$$
\rho_e - \rho_i \le  
c\, \nr \na^2 h \nr_{2,\Om}^{\frac{4}{N}-\frac{2(N-4)}{N (q-2)} }.
$$
Hence, \eqref{stability-SBT-C2-N-5} ensues  from \eqref{eq:inproofSBTpassaggioquasifinal}.
The case in which $\Ga$ is of class $C^{2,\ga}$ can be dealt similarly.
\par
To conclude the proof, we are left to show that \eqref{eq:SBT mappa Gauss stability} also holds.
To this aim, as done in the introduction, we observe that 
$$
\left|\nu(x)-\frac{x-z}{R}\right|\le 
\frac{|R-u_\nu(x)|+|\na h(x)|}{R} \ \mbox{ for } \ x\in\Ga.
$$
Hence, we infer that
\begin{equation*}\label{eq:vicinanza Gauss map per Serrin}
R \left(\int_\Ga\left|\nu(x)-\frac{x-z}{R}\right|^2 \frac{dS_x}{|\Ga|}\right)^{1/2}
\le \nr u_\nu-R\nr_{2,\Ga}+\nr \na h\nr_{2,\Ga} \le c \, \nr u_\nu-R\nr_{ 2 ,\Ga} ,
\end{equation*}
where we applied the triangle inequality and
the second inequality in
\eqref{trace-grad-h-serrin-deviation}. By using \eqref{bound-u-nu-curvature}, then \eqref{eq:SBT mappa Gauss stability} easily
follows from the last inequality above.
\end{proof}

\begin{rem}
\label{rem:old-stability}
{\rm
In order to compare the results of Theorem \ref{thm:SBT-improved-stability} to previous estimates, we recall what we obtained in \cite[Theorem 3.5]{MP3} --- the last up-to-date bound for stability in the Soap Bubble Theorem. In fact, there we obtained the bound
$$
\rho_e-\rho_i\le c\,\Psi\left(\nr H-H_0\nr_{L^2(\Ga)}\right),
$$
with
\begin{equation*}
\label{def-Psi}
\Psi(\si)=\begin{cases}
\si \ &\mbox{ if } \ N=2, 3, \\
\si^{1-\ve} \ &\mbox{ if } \ N=4, \\
\si^{2/(N-2)}  \ &\mbox{ if } \ N\ge 5,
\end{cases}
\end{equation*}
where the case $N=4$ must be interpreted thus: for any $0<\ve<1$, there exists a constant $c=c_\ve$ (which may blow up as $\ve\to 0$), such that case $N=4$ holds.
Theorem \ref{thm:SBT-improved-stability} clearly improves these profiles if $\Ga$ is either of class $C^2$ or $C^{2,\ga}$. Moreover, it also states that we can control linearly the deviation of the Gauss map from that of a sphere, at least in the $L^2$-norm.
}
\end{rem}

\smallskip

\section{Application to quantitative symmetry \\
in Serrin's overdetermined problem
}

In order to obtain stability estimates for Serrin's problem, we must use identity \eqref{identity-MP}. In fact, this relates the weighted integral at the right-hand side to the deviation $u_\nu-R$. Since the torsion $u$ can be easily bounded below by $\de_\Ga$ (see \cite[Lemma 3.1]{MP2}), we understand that  this time we must associate the difference $\rho_e - \rho_i$ with the weighted $L^2$-norm $\nr \de_\Ga^{1/2} \na^2 h \nr_{2, \Om}$. The following result goes in that direction.

\begin{thm}
	\label{thm:serrin-W22-stability} 
	Let $\Om\subset\RR^N$, $N\ge 2$, be a bounded domain with boundary $\Ga$ of class $C^{2}$
	and $z \in \Om$ be a global minimum point of the solution $u$ of \eqref{serrin1}.
	Consider the function  $h=Q^z-u$, with $Q^z$ given by \eqref{quadratic}.
	Then, there exists a constant $c = c (N, d_\Om, r_i , r_e)$ such that
	\begin{equation*}
		\label{W22-Serrin-stability-cases}
		\rho_e - \rho_i  \le 
		c\,\begin{cases}
			\nr \de_\Ga^{1/2} \na^2 h \nr_{2,\Om}   \ &\mbox{if $N=2$}; \\
			\displaystyle  \nr \de_\Ga^{1/2} \na^2 h \nr_{2,\Om} \max \left[ \log \left( \frac{e \, \nr \na h \nr_{ \infty, \Om}  }{  \nr \de_\Ga^{1/2} \na^2 h \nr_{2,\Om}  } \right) ,1 \right]  \ &\mbox{if $N=3$}; \vspace{3pt} \\
			\nr \na h \nr_{\infty, \Om}^{(N-3)/(N-1)} \nr \de_\Ga^{1/2} \na^2 h \nr_{2,\Om}^{2/(N-1)}   \ &\mbox{if $N \ge 4$.}
		\end{cases}
	\end{equation*}
\end{thm}
\begin{proof}
	(i) Let $N=2$.
	By using Lemma \ref{lem:intermediate h} with $p=4$ we have that
	$$
	\rho_e - \rho_i \le
	c 
	\, \nr \na h \nr_{4, \Om} .
	$$
	By applying Remark \ref{rem:weighted-grad} with $r=4$, $p=2$, and $\al= 1/2$, we obtain that
	\begin{equation*}
		\nr \na h \nr_{4, \Om} \le c \, \nr \de_\Ga^{1/2} \na^2 h  \nr_{2, \Om} ,
	\end{equation*}
	and the conclusion follows.
	
	(ii) Let $N=3$.
	By using Remark \ref{rem:weighted-grad} with $r=3$, $p=2$, $\al= 1/2$, we get
	$$
	\nr \na h \nr_{3 , \Om} \le c  \, \nr \de_\Ga^{1/2} \na^2 h \nr_{2,\Om}.
	$$
The conclusion follows by using Lemma \ref{lem:intermediate h} with $p=N=3$.
	
	(iii) When $N\ge 4$, we use Lemma \ref{lem:intermediate h} with $p=2N/(N-1)$ and put it together with Remark \ref{rem:weighted-grad} with $r={\frac{2N}{N-1}}$, $p=2$, $\al=1/2$.
\end{proof}

Theorem \ref{thm:serrin-W22-stability}  is sufficient to obtain a better stability bound for Serrin's problem when $N=3$.
Now, to gain better stability for $N\ge 4$, we need to obtain a bound similar to that in Corollary \ref{cor:bound-for-gradient}, but with $\nr \na^2 h \nr_{p,\Om}$ replaced by $\nr \de_\Ga^{1/2} \na^2 h \nr_{p,\Om}$.
 This time, we proceed differently.
\begin{lem}
		\label{lem:NEW_bound-for-gradient_using_old estimates}
	Set $1 \le p\le\infty$ and $q>N$.
 Let $\Om\subset\RR^N$, $N\ge 2$, be a bounded domain with boundary $\Ga$ of class $C^{2}$ and assume that $h\in W^{2,q}(\Om)$.
Then, there exists a constant $c = c (N, p,q)$ such that
		\begin{equation}\label{eq:NEW gradient bound for h using old estimates}
			\nr \na h \nr_{\infty, \Om}^{N+p\,(1-N/q)} \le
			c\,|\Om|\,\nr\na h \nr_{ p, \Om}^{p\, (1-N/q)} \nr \na^2 h\nr_{q, \Om}^N .
		\end{equation}
%
%
	\end{lem}
	\begin{proof}
For any $x\in\ol{\Om}$ there is a cone $\cC_{x,a}\subset\Om$. Applying \eqref{eq:MorreySobolev} with $p=q$ to any cone $\cC_{x,\si}\subset\cC_{x,a}$ gives that
$$
|f(x)|\le \int_{\cC_{x,\si}}|f|\,d\mu_y+c\,\si\,\| \na f \|_{q, \cC_{x,\si}}
\le \nr f\nr_{p,\cC_{x,\si}}+c\,\si\,\| \na f \|_{q, \cC_{x,\si}},
$$
where we used H\"older's inequality at the second inequality. Here, $c=c(N, q)$.
Thus, we have that
$$
\max_{\ol{\Om}} f - \min _{\ol{\Om}} f \le 2 \max_{\ol{\Om}}|f| \le
%
%
c\,\left(|\Om|^{1/p} \si^{-N/p}\nr f\nr_{p,\Om}+c\,|\Om|^{1/q}\si^{1-N/q}\,\| \na f \|_{q, \Om}\right),
$$
for every $\si\in (0,a)$, where in the second inequality we also used the monotonicity of Lebesgue's integral with respect to set inclusion.
Here, $c=c(N,p,q)$ (notice that the dependence on $\te$ can be dropped, since $\te=\sqrt{2}/2$, being as $\Ga$ of class $C^2$). We now minimize in $\si$ as done before. This time, we omit the details. We end up with the formula:
$$
\max_{\ol{\Om}} f - \min _{\ol{\Om}} f \le c\,|\Om|^{\frac1{N+p(1-N/q)}} \nr f\nr_{p,\Om}^{\frac{p (1-N/q)}{N+p(1-N/q)}}\nr\na f\nr_{q,\Om}^{\frac{N}{N+p(1-N/q)}}.
$$
This holds for any $x\in\ol{\Om}$ and $1\le p<q\le\infty$. 
By choosing $f$ as any directional derivative $h_\ell$ of $h$ and using that, with our choice of $z$, $|h_\ell  (x)|= | h_\ell(x) - h_\ell (z)|$, we thus get that
		$$
		| h_\ell (x)|^{N+p\,(1-N/q)} \le c\,|\Om|\,\nr h_\ell \nr_{ p, \Om}^{p\,(1-N/q)}
\nr \na h_\ell \nr_{q, \Om}^N
		\ \text{ for any } \ x \in \ol{\Om}.
		$$
\par
Hence, \eqref{eq:NEW gradient bound for h using old estimates} follows by observing that $|h_\ell|\le|\na h|$, $| \na h_\ell | \le | \na^2 h|$, and by choosing $\ell$  such that  $h_\ell (x) = |\na h (x)|$ and $x \in \Ga $ that maximizes $|\na h|$ on $\ol{\Om}$. 
\par
As for Corollary \ref{cor:bound-for-gradient}, the lemma could be stated for a domain satisfying an interior cone condition.
\end{proof}

\begin{cor}
\label{cor:nuovi bound gradiente per SBT e Serrin}
		Set $1\le p<2N$ and $q>N$. Under the assumptions of Lemma \ref{lem:NEW_bound-for-gradient_using_old estimates}, we have that
		%
		%
\begin{equation}
\label{eq:NEW_GRADIENTBOUNDFORSERRIN}
				\nr \na h \nr_{\infty, \Om}^{2N-p+2p\,(1-N/q)} \le
				c\, \nr \na^2 h \nr_{q,\Om}^{2N-p} \, \nr \de_\Ga^{ 1/2 } \na^2 h \nr_{p, \Om }^{2p\,(1-N/q)} .
			\end{equation}
Here, the constant $c$ only depends on $N$, $p$, $q$, $d_\Om$, $r_i$, and $r_e$. 
\end{cor}

\begin{proof}
We use Remark \ref{rem:weighted-grad} with $r$, $p$, and $\al$ replaced by $2pN/(2N-p)$, $p$, and $1/2$, respectively.
We thus get that
$$
\nr \na h \nr_{\frac{2pN}{2N-p} , \Om } \le c \, \nr \de_\Ga^{1/2} \na^2 h \nr_{p, \Om}.
$$
Therefore, \eqref{eq:NEW_GRADIENTBOUNDFORSERRIN} follows by combining this bound and 			\eqref{eq:NEW gradient bound for h using old estimates} with $p$ replaced by $2pN/( 2N-p )$.		
\end{proof}

\begin{thm}[Serrin's problem: enhanced stability]
\label{thm:Improved-Serrin-stability}
Let $\Om\subset\RR^N$, $N\ge2$, be a bounded domain with boundary $\Ga$ of class $C^{2}$ and set
$
R=N |\Om|/ | \Ga|.
$
\par
Let $u$ be the solution of problem \eqref{serrin1} and 
$z\in\Om$ be a global minimum point of $u$ in $\ol{\Om}$.
Then, 
there exists a constant $c = c( N, d_\Om, r_i, r_e)$ such that
$$
	\rho_e - \rho_i  \le 
	c\,\begin{cases}
	 \nr u_\nu - R \nr_{2,\Ga}   \ &\mbox{if $N=2$}; \\
		\displaystyle  \nr u_\nu - R \nr_{2,\Ga} \max \left[ \log \left( \frac{1  }{  \nr u_\nu - R \nr_{2,\Ga}  } \right) ,1 \right]  \ &\mbox{if $N=3$}.
	\end{cases}
$$
When $N\ge 4$, for any $q\in(N,\infty)$, there exists a constant $c= c( N,q, d_\Om, r_i, r_e)$ such that
\begin{equation}
	\label{eq:final-estimate Serrin for C2 rhoei}
\rho_e - \rho_i  \le c\,\nr u_\nu - R \nr_{2,\Ga}^{\frac{4-2N/q}{N+1-2N/q}}.
\end{equation}
Moreover (for any $N\ge 2)$,
\begin{equation}\label{eq:Serrin mappa Gauss stability}
	R\,\left\nr \nu-\frac{\na Q^z}{R}\right\nr_{2,\Ga} \le c \,  \nr u_\nu - R \nr_{2,\Ga},
\end{equation}
for some constant $c = c( N, d_\Om, r_i, r_e )$.
\par
If $\Ga$ is of class $C^{2,\ga}$, $0<\ga\le 1$, the stability exponent in \eqref{eq:final-estimate Serrin for C2 rhoei} for $N \ge 4$ can be replaced its limit as $q\to\infty$, i.e.  $4/(N+1)$.
In this case, $c$ only depends on $N$, $d_\Om$, and the $C^{2,\ga}$-modulus of continuity of $\Ga$.
\end{thm}

\begin{proof}
It is sufficient to notice that, thanks to \eqref{identity-MP} and the pointwise inequality $\de_\Ga\le -2u/r_i$, we can infer that
$$
\nr\de_\Ga^{1/2} \na^2 h\nr_{2,\Om}^2\le c\,\int_\Om(-u)\,|\na^2 h|^2 dx\le
c\,\nr u_\nu-R\nr_{2,\Ga}\,\nr h_\nu\nr_{2,\Ga}.
$$
Thus, by \eqref{trace-grad-h-serrin-deviation}, we obtain that
$$
\nr\de_\Ga^{1/2} \na^2 h\nr_{2,\Om}\le c\,\nr u_\nu-R\nr_{2,\Ga}.
$$
Therefore, with this inequality in hand, we can proceed similarly to the proof of Theorem \ref{thm:SBT-improved-stability} by also taking into account Remark \ref{rem:schauder-calderon-zygmund}. 
For instance,  the claim for $N\ge 4$ simply follows from  Theorem \ref{thm:serrin-W22-stability} and Corollary \ref{cor:nuovi bound gradiente per SBT e Serrin} with $p=2$.
\par
The remaining claims follow from Theorem \ref{thm:serrin-W22-stability} at once.
\end{proof}

\begin{rem}
{\rm
In order to compare the results of Theorem \ref{thm:Improved-Serrin-stability} to previous estimates, it is sufficient to recall what we obtained in \cite[Theorem 3.1]{MP3} --- the last up-to-date bound for stability in Serrin's problem. In fact, there we obtained the bound
$$
\rho_e-\rho_i\le c\,\Psi\left(\nr u_\nu-R\nr_{L^2(\Ga)}\right),
$$
with
\begin{equation*}
\label{def-Psi}
\Psi(\si)=\begin{cases}
\si \ &\mbox{ if } \ N=2, \\
\si^{1-\ve} \ &\mbox{ if } \ N=3, \\
\si^{2/(N-1)}  \ &\mbox{ if } \ N\ge 4.
\end{cases}
\end{equation*}
The case $N=3$ must be interpreted thus: for any $0<\ve<1$ there exists a constant $c=c_\ve$ (which may blow up as $\ve\to 0$), such that case $N=3$ holds.
\par
The comparison with Theorem \ref{thm:Improved-Serrin-stability} is left to the reader.
\par
As already mentioned in the introduction for the Soap Bubble Theorem, if one adopts a stronger norm for the deviation $u_\nu-R$, linear stability can also be obtained in general dimension. See for instance \cite{GO}.
 }
\end{rem}

\begin{rem}
{\rm
A tedious inspection of the corresponding proofs tells us that the dependence of the relevant constant $c$ on the parameter $r_e$ can be removed whenever $\Ga$ is mean convex.
Notice that, in this case, the bound in \eqref{bound-M} can be replaced by \cite[Formula (2.4)]{MP3}.
}
\end{rem}

\section*{Acknowledgements}
Rolando Magnanini was partially supported by the Gruppo Nazionale per l'Analisi Matematica, la Probabilit\`a e le loro Applicazioni (GNAMPA) of the italian Istituto Nazionale di Alta Matematica (INdAM). 
Giorgio Poggesi is supported by the Australian Laureate Fellowship FL190100081 ``Minimal surfaces, free boundaries and partial differential equations'' and is member of AustMS and INdAM/GNAMPA.

\end{document}